\documentclass{amsart}
\usepackage{amsfonts}

\newtheorem{theorem}{Theorem}
\theoremstyle{plain}

\numberwithin{equation}{section}
\input{tcilatex}

\begin{document}
\title[The Pell equations $x^{2}-\left( a^{2}+2a\right) y^{2}=N$]{Solutions
of the Pell equation $x^{2}-\left( a^{2}+2a\right) y^{2}=N$ via generalized
Fibonacci and Lucas numbers}
\author{Bilge PEKER}
\address{Department of Mathematics Education, Ahmet Kelesoglu Education
Faculty, Necmettin Erbakan University, Konya, Turkey.}
\email{bilge.peker@yahoo.com}
\date{March 6, 2013}

\begin{abstract}
In this study, we find continued fraction expansion of $\sqrt{d}$ when $%
d=a^{2}+2a$ where $a$ is positive integer. We consider the integer solutions
of the Pell equation $x^{2}-\left( a^{2}+2a\right) y^{2}=N$ when $N\in
\left\{ \pm 1,\pm 4\right\} $. We formulate the $n$-th solution $\left(
x_{n},y_{n}\right) $ by using the continued fraction expansion. We also
formulate the $n$-th solution $\left( x_{n},y_{n}\right) $ via the
generalized Fibonacci and Lucas sequences.
\end{abstract}

\maketitle

\bigskip \textbf{1. Introduction and Preliminaries}

\bigskip The equation $x^{2}-dy^{2}=N,$ with given integers $d$ and $N,$
unknowns $x$ and $y$, is called as Pell equation. In the literature, there
are several methods for finding the integer solutions of Pell equation such
as the Lagrange-Matthews-Mollin algorithm, the cyclic method, Lagrange's
system of reductions, use of binary quadratic forms, etc.

If $d$ is negative, the equation can have only a finite number of solutions.
If $d$ is a perfect square, i.e. $d=a^{2},$ the equation reduces to $\left(
x-ay\right) \left( x+ay\right) =N$ and there is only a finite number of
solutions. If $d$ is a positive integer but not a perfect square, then
simple continued fractions are very useful. The simple continued fraction
expansion of $\sqrt{d}\ $has the form $\sqrt{d}=\left[ a_{0},\overline{%
a_{1},a_{2},a_{3},...,a_{m-1},2a_{0}}\right] $ with $a_{0}=\left[ \sqrt{d}%
\right] $. If the fundamental solution of $x^{2}-dy^{2}=1$ is $x=x_{1}$ and $%
y=y_{1}$, then all nontrivial solutions are given by $x=x_{n}$ and $y=y_{n}$%
, where $x_{n}+y_{n}\sqrt{d}=\left( x_{1}+y_{1}\sqrt{d}\right) ^{n}$. If a
single solution $\left( x,y\right) =\left( g,h\right) $ of the equation $%
x^{2}-dy^{2}=N$ is known, other solutions can be found. Let $\left(
r,s\right) $ be a solution of the unit form $x^{2}-dy^{2}=1$. Then $\left(
x,y\right) =\left( gr\pm dhs,gs\pm hr\right) $ are solutions of the equation 
$x^{2}-dy^{2}=N$.

Given a continued fraction expansion of $\sqrt{d}$, where all the $a_{i}$'s
are real and all except possibly $a_{0}$ are positive, define sequences $%
\left\{ p_{n}\right\} $ and $\left\{ q_{n}\right\} $ by $p_{-2}=0$, $%
p_{-1}=1 $, $p_{k}=a_{k}p_{k-1}+p_{k-2}$ and $q_{-2}=1$, $q_{-1}=0$, $%
q_{k}=a_{k}q_{k-1}+q_{k-2}$ for $k\geq 0$. Let $m$ be the length of the
period of continued fraction. Then the fundamental solution of $%
x^{2}-dy^{2}=1$ is%
\begin{equation*}
\left( x_{1},y_{1}\right) =\left\{ 
\begin{array}{c}
\left( p_{m-1},q_{m-1}\right) \\ 
\left( p_{2m-1},q_{2m-1}\right)%
\end{array}%
\right. 
\begin{array}{c}
\text{if }m\text{ is even} \\ 
\text{if }m\text{ is odd.}%
\end{array}%
\end{equation*}

\bigskip\ If the length of the period of continued fraction is even, then
the equation $x^{2}-dy^{2}=-1$ has no integer solutions. If $m$ is odd, the
fundamental solution of $x^{2}-dy^{2}=-1$ is given by $\left(
x_{1},y_{1}\right) =\left( p_{m-1},q_{m-1}\right) $ $\left[ 4\right] $.

Now let us give the following known theorems $\left[ 5\right] $ that will be
needed for the next section.

\begin{theorem}
Let $d\equiv 2\left( \func{mod}4\right) $ or $d\equiv 3\left( \func{mod}%
4\right) $. Then the equation $x^{2}-dy^{2}=-4$ has positive integer
solutions if and only if the equation $x^{2}-dy^{2}=-1$ has positive integer
solutions.
\end{theorem}

\begin{theorem}
Let $d\equiv 0\left( \func{mod}4\right) $. If fundamental solution of the
equation $x^{2}-\left( d/4\right) y^{2}=1$ is $x_{1}+y_{1}\sqrt{d/4}$, then
fundamental solution of the equation $x^{2}-dy^{2}=4$ is $\left(
2x_{1},y_{1}\right) $.
\end{theorem}

\begin{theorem}
Let $d\equiv 1\left( \func{mod}4\right) $ or $d\equiv 2\left( \func{mod}%
4\right) $ or $d\equiv 3\left( \func{mod}4\right) $. If fundamental solution
of the equation $x^{2}-dy^{2}=1$ is $x_{1}+y_{1}\sqrt{d}$, then fundamental
solution of the equation $x^{2}-dy^{2}=4$ is $\left( 2x_{1},2y_{1}\right) $.
\end{theorem}

The next two theorems can be found in $\left[ 2\right] $ and $\left[ 5\right]
.$

\begin{theorem}
Let $x_{1}+y_{1}\sqrt{d}$ be the fundamental solution of the equation $%
x^{2}-dy^{2}=4$. Then all positive integer solutions of the equation $%
x^{2}-dy^{2}=4$ are given by%
\begin{equation*}
x_{n}+y_{n}\sqrt{d}=\left( x_{1}+y_{1}\sqrt{d}\right) ^{n}/2^{n-1}
\end{equation*}%
with $n\geq 1$.
\end{theorem}

\begin{theorem}
Let $x_{1}+y_{1}\sqrt{d}$ be the fundamental solution of the equation $%
x^{2}-dy^{2}=-4$. Then all positive integer solutions of the equation $%
x^{2}-dy^{2}=-4$ are given by%
\begin{equation*}
x_{n}+y_{n}\sqrt{d}=\left( x_{1}+y_{1}\sqrt{d}\right) ^{2n-1}/2^{2n-2}
\end{equation*}%
with $n\geq 1$.
\end{theorem}

In this study $\left[ 3\right] $, since generalized Fibonacci and Lucas
sequences related solutions of the forthcoming Pell equation are going to be
taken into consideration, let us briefly recall the generalized Fibonacci
sequences $\left\{ U_{n}\left( k,s\right) \right\} $ and Lucas sequences $%
\left\{ V_{n}\left( k,s\right) \right\} $. Let $k$ and $s$ be two non-zero
integers with $k^{2}+4s>0$. Generalized Fibonacci sequence is defined by%
\begin{equation*}
U_{0}\left( k,s\right) =0,U_{1}\left( k,s\right) =1
\end{equation*}%
and%
\begin{equation*}
U_{n+1}\left( k,s\right) =kU_{n}\left( k,s\right) +sU_{n-1}\left( k,s\right) 
\end{equation*}%
for $n\geq 1$. Generalized Lucas sequence is defined by%
\begin{equation*}
V_{0}\left( k,s\right) =2,V_{1}\left( k,s\right) =k
\end{equation*}%
and%
\begin{equation*}
V_{n+1}\left( k,s\right) =kV_{n}\left( k,s\right) +sV_{n-1}\left( k,s\right) 
\end{equation*}%
for $n\geq 1$. It is also well-known from the literature that generalized
Fibonacci and Lucas numbers have many interesting and significant
properties. Binet's formulas are probably the most important one among them.
For generalized Fibonacci and Lucas sequences, Binet's formulas are given by 
$U_{n}\left( k,s\right) =\frac{\alpha ^{n}-\beta ^{n}}{\alpha -\beta }$ and $%
V_{n}\left( k,s\right) =\alpha ^{n}+\beta ^{n}$ where $\alpha =\left( k+%
\sqrt{k^{2}+4s}\right) /2$ and $\beta =\left( k-\sqrt{k^{2}+4s}\right) /2$ $%
\left[ 7\right] $.

There are a large number of studies concerning Pell equation in the
literature. G\"{u}ney $\left[ 1\right] $ solved the Pell equations $%
x^{2}-\left( a^{2}b^{2}+2b\right) y^{2}=N$ when $N\in \left\{ \pm 1,\pm
4\right\} $.

In this study, we consider the integer solutions of the Pell equation 
\begin{equation*}
x^{2}-\left( a^{2}+2a\right) y^{2}=N
\end{equation*}%
in terms of the generalized Fibonacci and Lucas numbers.

\bigskip

\textbf{2. Main Results}

We consider the integer solutions of the Pell equation%
\begin{equation}
E:x^{2}-\left( a^{2}+2a\right) y^{2}=1.  \tag{1}
\end{equation}

\begin{theorem}
Let $E$ be the Pell equation in $\left( 1\right) $. Then the followings hold:

$\left( \mathbf{i}\right) $ The continued fraction expansion of $\sqrt{%
a^{2}+2a}$ is 
\begin{equation*}
\sqrt{a^{2}+2a}=\left[ a;\overline{1,2a}\right] \text{.}
\end{equation*}

$\left( \mathbf{ii}\right) $ The fundamental solution is%
\begin{equation*}
\left( x_{1},y_{1}\right) =\left( a+1,1\right)
\end{equation*}

$\left( \mathbf{iii}\right) $ The n-th solution $\left( x_{n},y_{n}\right) $
can be find by%
\begin{equation*}
\frac{x_{n}}{y_{n}}=\left[ a;\left( 1,2a\right) _{n-1},1\right]
\end{equation*}%
where $\left( 1,2a\right) _{n-1}$ means that there are $n-1$ successive
terms $\left( 1,2a\right) $.\qquad \qquad \qquad \qquad \qquad

\begin{proof}
$\left( \mathbf{i}\right) $ 
\begin{equation*}
\begin{array}{c}
\sqrt{a^{2}+2a}=a+\left( \sqrt{a^{2}+2a}-a\right) =a+\frac{1}{\frac{\sqrt{%
a^{2}+2a}+a}{2a}} \\ 
\text{ \ \ \ \ \ \ \ \ \ }=a+\frac{1}{1+\frac{\sqrt{a^{2}+2a}-a}{2a}}=a+%
\frac{1}{1+\frac{1}{\sqrt{a^{2}+2a}+a}} \\ 
=a+\frac{1}{1+\frac{1}{2a+\left( \sqrt{a^{2}+2a}-a\right) }}.\text{ \ \ \ \
\ }%
\end{array}%
\end{equation*}

Therefore $\sqrt{a^{2}+2a}=\left[ a;\overline{1,2a}\right] $. This completes
the proof.

$\left( \mathbf{ii}\right) $ The period length of the continued fraction
expansion of $\sqrt{a^{2}+2a}$ is 2. Therefore, the fundamental solution of
the equation $x^{2}-\left( a^{2}+2a\right) y^{2}=1$ is $p_{1}+q_{1}\sqrt{%
a^{2}+2a}$. It is easily seen that $p_{1}=a_{1}p_{0}+p_{-1}=a+1$ and $q_{1}=$
$a_{1}q_{0}+q_{-1}=1$. That is, the fundamental solution of $x^{2}-\left(
a^{2}+2a\right) y^{2}=1$ is $\left( x_{1},y_{1}\right) =\left( a+1,1\right) $%
.

$\left( \mathbf{iii}\right) $ For $n=1$, we get $\frac{x_{1}}{y_{1}}=\left[
a;1\right] =a+\frac{1}{1}=\frac{a+1}{1}$. Hence it is true for $n=1$.

We assume that $\left( x_{n},y_{n}\right) $ is a solution of $x^{2}-\left(
a^{2}+2a\right) y^{2}=1.$ That is, $\frac{x_{n}}{y_{n}}=[a;(1,2a)_{n-1},1]$%
\text{.}

Now we must show that it holds for $\left( x_{n+1},y_{n+1}\right) $.%
\begin{equation*}
\begin{array}{c}
\frac{x_{n+1}}{y_{n+1}}=a+\frac{1}{1+\frac{1}{2a+\frac{1}{1+\frac{1}{2a+%
\frac{1}{...2a+1}}}}}\text{ \ \ \ \ } \\ 
\text{\ \ \ \ \ \ \ } \\ 
\text{ \ \ \ }=a+\frac{1}{1+\frac{1}{a+a+\frac{1}{1+\frac{1}{2a+\frac{1}{%
...2a+1}}}}} \\ 
\\ 
=a+\frac{1}{1+\frac{1}{a+\frac{x_{n}}{y_{n}}}}\text{ \ \ \ \ \ \ \ \ } \\ 
\text{\ \ } \\ 
\text{ }=\frac{\left( a+1\right) x_{n}+\left( a^{2}+2a\right) y_{n}}{%
x_{n}+\left( a+1\right) y_{n}}.\text{ \ }%
\end{array}%
\end{equation*}

$\left( x_{n+1},y_{n+1}\right) $ is a solution of $x^{2}-\left(
a^{2}+2a\right) y^{2}=1$ since $x_{n+1}^{2}-\left( a^{2}+2a\right)
y_{n+1}^{2}=\left( \left( a+1\right) x_{n}+\left( a^{2}+2a\right)
y_{n}\right) ^{2}-\left( a^{2}+2a\right) \left( x_{n}+\left( a+1\right)
y_{n}\right) ^{2}=x_{n}^{2}-\left( a^{2}+2a\right) y_{n}^{2}=1.$
\end{proof}
\end{theorem}

\begin{theorem}
All positive integer solutions of the equation $x^{2}-\left( a^{2}+2a\right)
y^{2}=1$ are given by%
\begin{equation*}
\left( x_{n},y_{n}\right) =\left( \left( V_{n}\left( 2a+2,-1\right) \right)
/2,U_{n}\left( 2a+2,-1\right) \right)
\end{equation*}%
with $n\geq 1$.
\end{theorem}

\begin{proof}
By Theorem 6-ii, all positive integer solutions of the equation $%
x^{2}-\left( a^{2}+2a\right) y^{2}=1$ are given by%
\begin{equation*}
x_{n}+y_{n}\sqrt{a^{2}+2a}=\left( a+1+\sqrt{a^{2}+2a}\right) ^{n}
\end{equation*}%
with $n\geq 1$. Assume that $\alpha =a+1+\sqrt{a^{2}+2a}$ and $\beta =a+1-%
\sqrt{a^{2}+2a}$. Then $\alpha -\beta =2\sqrt{a^{2}+2a}$.%
\begin{equation*}
x_{n}+y_{n}\sqrt{a^{2}+2a}=\alpha ^{n}
\end{equation*}%
and%
\begin{equation*}
x_{n}-y_{n}\sqrt{a^{2}+2a}=\beta ^{n}.
\end{equation*}

Therefore $x_{n}=\frac{\alpha ^{n}+\beta ^{n}}{2}=\frac{V_{n}\left(
2a+2,-1\right) }{2}$ and $y_{n}=\frac{\alpha ^{n}-\beta ^{n}}{2\sqrt{a^{2}+2a%
}}=\frac{\alpha ^{n}-\beta ^{n}}{\alpha -\beta }=U_{n}\left( 2a+2,-1\right)
. $ That is, $\left( x_{n},y_{n}\right) =\left( \frac{V_{n}\left(
2a+2,-1\right) }{2},U_{n}\left( 2a+2,-1\right) \right) .$
\end{proof}

\begin{theorem}
The Pell equation $x^{2}-\left( a^{2}+2a\right) y^{2}=-1$ has no positive
integer solutions.
\end{theorem}

\begin{proof}
The length of the period of continued fraction $\sqrt{a^{2}+2a}$ is $2$,
that is even, then the equation $x^{2}-\left( a^{2}+2a\right) y^{2}=-1$ has
no positive integer solutions.
\end{proof}

\begin{theorem}
The fundamental solution of the Pell equation $x^{2}-\left( a^{2}+2a\right)
y^{2}=4$ is%
\begin{equation*}
\left( x_{1},y_{1}\right) =\left( 2a+2,2\right) \text{.}
\end{equation*}
\end{theorem}

\begin{proof}
It is obvious from Theorem 3 and Theorem 6-ii.
\end{proof}

\begin{theorem}
All positive integer solutions of the equation $x^{2}-\left( a^{2}+2a\right)
y^{2}=4$ are given by%
\begin{equation*}
\left( x_{n},y_{n}\right) =\left( V_{n}\left( 2a+2,-1\right) ,2U_{n}\left(
2a+2,-1\right) \right)
\end{equation*}%
with $n\geq 1$.
\end{theorem}

\begin{proof}
Using Theorem 4 and Theorem 9, all positive integer solutions of the
equation $x^{2}-\left( a^{2}+2a\right) y^{2}=4$ are given by%
\begin{equation*}
x_{n}+y_{n}\sqrt{a^{2}+2a}=\left( 2a+2+2\sqrt{a^{2}+2a}\right)
^{n}/2^{n-1}=2\left( \left( 2a+2+2\sqrt{a^{2}+2a}\right) /2\right) ^{n}
\end{equation*}%
with $n\geq 1$. Assume that $\alpha =\left( 2a+2+2\sqrt{a^{2}+2a}\right) /2$
and $\beta =\left( 2a+2-2\sqrt{a^{2}+2a}\right) /2$. Then $\alpha -\beta =2%
\sqrt{a^{2}+2a}$.%
\begin{equation*}
x_{n}+y_{n}\sqrt{a^{2}+2a}=2\alpha ^{n}
\end{equation*}%
and%
\begin{equation*}
x_{n}-y_{n}\sqrt{a^{2}+2a}=2\beta ^{n}.
\end{equation*}

Therefore $x_{n}=\alpha ^{n}+\beta ^{n}=V_{n}\left( 2a+2,-1\right) $ and $%
y_{n}=\frac{\alpha ^{n}-\beta ^{n}}{\sqrt{a^{2}+2a}}=2\frac{\alpha
^{n}-\beta ^{n}}{\alpha -\beta }=2U_{n}\left( 2a+2,-1\right) .$ That is, $%
\left( x_{n},y_{n}\right) =\left( V_{n}\left( 2a+2,-1\right) ,2U_{n}\left(
2a+2,-1\right) \right) .$
\end{proof}

\begin{theorem}
Let $a>2$. The Pell equation $x^{2}-\left( a^{2}+2a\right) y^{2}=-4$ has no
positive integer solutions.
\end{theorem}

\begin{proof}
If $a$ is odd, then $a^{2}+2a\equiv 3\left( \func{mod}4\right) $. From
Theorem 1, we know that the equation $x^{2}-\left( a^{2}+2a\right) y^{2}=-4$
has positive integer solutions if and only if the equation $x^{2}-\left(
a^{2}+2a\right) y^{2}=-1$ has positive integer solutions. But, from Theorem
8, the equation $x^{2}-\left( a^{2}+2a\right) y^{2}=-1$ has no positive
integer solutions. Therefore, $x^{2}-\left( a^{2}+2a\right) y^{2}=-4$ has no
positive integer solutions.

If $a$ is even, then $a^{2}+2a$ is even. Assume by way of contradiction that
there are positive integers $m$ and $n$ such that $m^{2}-\left(
a^{2}+2a\right) n^{2}=-4$. $a$ and $a^{2}+2a$ are even. Therefore, $m$ is
even. Let $a=2k$. Then $m^{2}-\left( 4k^{2}+4k\right) n^{2}=-4$ and we get $%
\left( m/2\right) ^{2}-\left( k^{2}+k\right) n^{2}=-1$.

The continued fraction expansion of $\sqrt{k^{2}+k}$ is%
\begin{equation*}
\sqrt{k^{2}+k}=\left\{ 
\begin{array}{c}
\left[ 1;\overline{2}\right] \\ 
\left[ k;\overline{2,2k}\right] \text{\ \ \ \ }%
\end{array}%
\right. ,%
\begin{array}{c}
\text{if }k=1 \\ 
\text{if }k>1\text{.}%
\end{array}%
\end{equation*}

We know that $a>2$. Therefore, $k>1$. Thus, the length of the period of
continued fraction $\sqrt{k^{2}+k}$ is $2$, that is even, then the equation $%
\left( m/2\right) ^{2}-\left( k^{2}+k\right) n^{2}=-1$ has no positive
integer solutions. So this is a contradiction. Then the equation $%
x^{2}-\left( a^{2}+2a\right) y^{2}=-4$ has no positive integer solutions.
\end{proof}

\bigskip \textbf{Acknowledgement. }\textit{This research is supported by
TUBITAK (The Scientific and Technological Research Council of Turkey) and
Necmettin Erbakan University Scientific Research Project Coordinatorship
(BAP). This study is a part of the corresponding author's Ph.D. Thesis.}

\end{document}